\documentclass[11pt]{amsart}
 \usepackage{amssymb,amsmath,latexsym,amsthm}

 \usepackage{times}
 \usepackage{dsfont}
 \usepackage[all]{xy}

 \usepackage[english,francais]{babel}

 \usepackage{url} 
 \usepackage{hyperref}

 \usepackage{color}

 \definecolor{greenbf}{rgb}{0, 0.7 ,0.3}
 
 \hypersetup{	
colorlinks=true, 
breaklinks=true, 
urlcolor= blue, 
linkcolor= red,	
citecolor= {greenbf},	
}

\newtheorem{theorem}{{Theorem}}[section]
\newtheorem{proposition}[theorem]{{Proposition}}
\newtheorem{isom.ext}[theorem]{{Trivial isometric extension}}
\newtheorem{definition}[theorem]{{Definition}}
\newtheorem{lemma}[theorem]{{Lemma}}
\newtheorem{corollary}[theorem]{{Corollary}}
\newtheorem{fact}[theorem]{{\sc Fact}}
\newtheorem{remark}[theorem]{{Remark}}
\newtheorem{example}[theorem]{{Example}}

\newtheorem{geod.ext}[theorem]{{Geodesic extension}}

\def\SO{{\sf{SO}}}
 
\def\SL{{\sf{SL}}}

\input amssym.def
\input amssym.tex

\begin{document}

\selectlanguage{english}



\title{ Geometry of warped products}


\author[A. Zeghib]{Abdelghani Zeghib }
\address{UMPA, CNRS, 
\'Ecole Normale Sup\'erieure de Lyon\hfill\break\indent
46, all\'ee d'Italie
69364 LYON Cedex 07, FRANCE}
\email{abdelghani.zeghib@ens-lyon.fr 
\hfill\break\indent
\url{http://www.umpa.ens-lyon.fr/~zeghib/}}
\date{\today}
\maketitle



\begin{abstract}  

This is a survey on the geometry  of warped products, without, or essentially 
with only soft, calculation. Somewhere in the paper, the goal was to give a 
synthetic
 account since existing approaches are rather analytic. 
Somewhere else,
we have interpreted statements, 
especially by means 
of 
a physical terminology.  This is essentially heuristic, but we think
it might be   helpful  in both directions, that is, in 
 going from a synthetic geometrical
language to a relativistic one, and vice-versa.

\end{abstract}

\tableofcontents

\section{Introduction} 

The warped product is 
  a  construction 
in the class of
pseudo-Riemannian manifolds generalizing the direct product, and
defined as follows.  
Let   $(L, h)$ and 
$(N, g)$ be   two pseudo-Riemannian 
  manifolds 
 and $ w: L \to  {\Bbb R}^+-\{0\}$
a {\it warping} function.  
  The warped
product \mbox{$M= L\times_wN$}, is the topological
product 
  $L \times N$, endowed with
the  metric  
$h \bigoplus w g$.
The   metric on 
$M$ will be 
usually denoted by  $\langle,\rangle $.
Here, we will be  especially interested in the case where $M$ 
is Lorentzian (a spacetime) and sometimes Riemannian.

\subsubsection*{Previous works.}  There are several references 
on warped products, we mention a few:
 \cite{A-B, 
BEP, C-C,Gau, O'n, Pra}. Some of them are, like the present 
one,  surveys,  but, in general, they originate from 
different points of view. The author met the subject while 
  working on \cite{Ze.gafa}
   \footnote{The present article exists in fact since 1999,  it  was destined to be published in the proceeding of a conference on pseudo-Riemannian geometry. 
 I came back recently around  the subject and discovered   interest of   some people who quote it; that is why I estimated it is  worthwhile to revive this paper. }.



\subsubsection*{Terminology.} Usually, $M$ is seen as a  bundle over
 $L$ (the {\it basis}) with {\it fiber} $N$. 
This point of view is surely justified. However,
it turns out that one sometimes needs to 
project onto $N$. 
Indeed, the distortion of the structure comes 
from the transverse structure of the foliation 
determined by the factor $L$, the study of which involves projecting 
on $N$ (see \S \ref{local.product}). Here, motivated by the  analogy with 
a group theoretic situation (justified by \ref{isometry.extension}), and to 
emphasize  its 
importance, we will call 
 $N$   the {\bf normal} 
factor 
of the warped product.

Let us introduce another useful  terminology
in this text. A warped product $M = L \times_w N$
is called a (global) {\bf generalized  Robertson-Walker} ({\bf GRW} in short) space, 
provided  $N$ is a pseudo-Riemannian manifold of 
constant curvature (see \cite{San} for 
another use of this terminology). Recall that classical 
Robertson-Walker spaces correspond to the case where 
$N$ is a Riemannian manifold of constant curvature, and
$L$ is an interval  of ${\Bbb R}$ endowed with the metric 
$-dt^2$.
These Lorentz spacetimes model an expanding universe.

\subsubsection*{Interests.} The warped product construction has at least 
two interesting properties. 
Firstly, it has a  practical interest, since 
it  gives 
sophisticated examples  from simple
 ones: calculation
on warped products 
is easy (but non-trivial).  Secondly, 
  having a large symmetry group 
generally  involve a warped product structure. 
Actually, being ``simple'' and  having a 
large  symmetry group, are    criterion of beauty.
Therefore, imposing a warped product structure 
is somewhat a formulation of a philosophical 
and an  aesthetical principle.

\subsection{Two  fundamental  extension facts.}
As in the case of  direct products, warped products enjoy
the two following properties:

$\bullet $ Dynamical property: extension of isometries.

$\bullet$ 
Geometric (static)  property: extension of
 geodesic submanifolds.

In the present  
article, we will specially 
 investigate the first point. (We hope 
to consider the second one  in a subsequent 
paper).


 Let $f: N \to N$ be  a diffeomorphism.  
Consider the  {\it trivial} (or product) extension:  $$\bar{f}: (x, y) \in
L \times N \to (x, f(y)) \in L\times N$$ 
With the notations above, 
we have 
$ \bar{f}^*(h \bigoplus wg) = h \bigoplus w f^*g$.
In particular:  

\begin{isom.ext}
\label{isometry.extension}
The trivial extension $\bar{f}$ is 
an isometry of 
$L\times_w N$ iff $f$ is an isometry of $N$.
\end{isom.ext}

  Warped products  are reminiscent of
semi-direct products in the category of groups, 
the factor $N$ playing the role of the normal subgroup.  
Indeed, Isom$(N)$ is a normal subgroup  of 
Isom$^\times(L \times_w N)$, which designs the group
of isometries of 
$L\times_wN$  preserving the topological product. 
This justifies 
 calling  $N$   the {\bf normal} factor
of the warped product. \\

The  following  is the second extension fact 
which  will be proved in \S \ref{transverse}.

\begin{geod.ext} 
\label{geodesic.extension}
Let $M = L \times_wN$ be a warped product, and $S $
 a submanifold of $N$. Then $S$ is geodesic in $N$ iff 
$L \times S $ is geodesic in $M$.

\end{geod.ext}

As a corollary, we obtain that a warped product has 
many non-trivial (i.e. with dim $>1$) geodesic submanifolds.
This is the  starting point of  rigidity of 
GRW spaces.


\subsection{Content and around the article.} The article contains personnel
points of view  rather than a standard survey on warped products.
 One fact which seems to be new in our 
approach here, is to consider local warped product structures, a notion which
belongs to the domain of foliations. This leads us in this paper to 
fix some known and used characterizations   (but sometimes difficult to find
in literature) of foliations with some transverse or tangential geometric 
structures (geodesic, umbilical, transversally pseudo-Riemannian...).

In another direction, one may also consider analytic pseudo-Riemannian manifolds, 
with a somewhere defined warped product structure, i.e. admitting an open set which
 is a warped product. 
In the direct (non-warped) product case, an analytic continuation
is easily defined in the whole universal cover. (The reason is that we get parallel 
plane fields which we extend by parallel transport).

This is no longer true 
in the warped case.  Firstly, in general, there is no mean to ``extend analytically''
(somewhere defined) foliations, since this is not uniquely defined even in the simply
 connected case, 
and also, because this would at most give rise to singular objects. 

In the case of a
somewhere defined    warped
product structure, we have a kind of a ``rigid geometric structure'', and one may use 
it as a model. One then considers points admitting charts isometric to it.
We will meet in \S \ref{proof.theorem} a situation where the technical realization of this idea works well.

Actually, one solves Einstein equations (i.e. spacetimes with some 
geometry) in charts, which are, thanks to reasonable symmetry hypotheses, endowed 
with a warped product structure.  One, in general, observes singularity of the metric
written in these co-ordinates systems.  It is usual to call such ``singularities'' {\bf inessential}.
From our point of view, they are still singularities, but for the warped
product structure. So, it is an interersting and natural problem 
to study the behaviour of analytic extension of  somewhere defined warped product: their
degenerations (horizons!) and their  regenerations (but  in a different  physical 
nature). That is a question that the present  article would suggest to consider
and study in a systematic way, however,  we do here only a few in the  particular case 
developed in  \S \ref{proof.theorem}.


\subsection{Preliminary examples}

\subsubsection{Polar coordinates.}
\label{polar.1}
This example illustrates how the presence of a 
warped product structure is related to symmetry, and how then, 
it is useful, 
as are the polar coordinates. Let us start with $M^n$ a Riemannian
 manifold, and
let $x \in M$.
Locally 
$M-\{ x \}$ is isometric to 
${\Bbb R}^+ \times S^{n-1}$, endowed with 
a metric 
$g = dr^2 \bigoplus  g_r$, where $g_r$ is a metric on 
$S^{n-1}$. 
Observe that 
$O(n)$ acts naturally by $(A.(r,u)) \to (r, A(u))$.

\begin{fact}
Polar coordinates determine 
 a warped product, that is, there is a metric $g$ on $S^{n-1}$ 
and a function $w(r)$ such that $g_r = w(r)g$, 
iff, 
the natural  action of $O(n)$ is isometric. It then follows
that   
 $g$ is, up to a multiplicative
factor,  the canonical metric on 
$S^{n-1}$, and that all the 2-planes at $x$ have the same
sectional curvature.

\end{fact}

\begin{proof} Assume we have a warped product.
In order to prove that the $O(n)$-action is isometric, it suffices to 
show that it is isometric on each sphere $S_r= \{r\} \times S^{n-1}$. Let 
$A \in O(n)$.
All these spheres are homothetic, and the metric distortion 
of $A$ is the same on all of them. But this distortion tends 
to 1 when 
$r \to 0$. Therefore, $A$ has distortion 1 on each $S_r$, that 
is 
$A$ acts isometrically. The remaining part  of the fact is standard.

\end{proof}

For example, polar coordinates determine 
a warped product
in the case of constant curvature Riemannian spaces, 
 the Euclidean case corresponds to 
${\Bbb R}^+ \times_{r^2} S^{n-1}$.

The previous fact 
generalizes to pseudo-Riemannian manifolds. More precisely, the polar
 coordinates
at a point $x$ of a pseudo-Riemannian  manifold $M^{p,q}$ of 
type
$(p,q)$,  give rise to a warped product 
structure, iff, the natural action of 
$O(p, q)$ is isometric.
Let us  call $x$ in this case, a point of {\it complete symmetry}.
All the non-degenerate 2-planes 
at such a  point 
  have the same sectional curvature.

In particular, if 
all the points of $M$ are points of  complete 
symmetry,
then, $M$ has a constant curvature.
It is then   natural to ask if there are non-trivial, 
i.e. with non constant curvature,  examples
of pseudo-Riemannian manifolds with at least
one  point of complete symmetry.
 An averaging method works to give examples,  in
 the Riemannian case, since $O(n)$ is compact. 
In the other cases, the ``spheres''
become 
complicated, and 
 a large isotropy group
at some point, 
may create extra symmetry 
elsewhere. However, nontrivial  
examples do exist, for instance, any Lorentz metric 
on ${\Bbb R}^2$ of the form 
$F(xy) dx dy$, where $F$ is 
a positive real  function defined on an interval 
containing 0,  admits $(0,0)$ as a
point of complete symmetry. (The metric is 
defined on an open subset of ${\Bbb R}^2$
delimited by hyperbolas
$xy = $ constant).
 A  celebrated example  of this form is the   Kruskal 
plane (see for instance \cite{O'n}). 

More generally, in any dimension, one may consider Lorentz metrics of the form 
$g= F(q)q$ where $q$ is a Lorentz form. The origin is a point of complete 
symmetry for $g$. Let us however
 that the situation becomes really rigid if one asks 
for many points of complete symmetry.

\subsubsection{Riemannian symmetric spaces.} 
We find the representation of 
the hyperbolic (Riemannian) space ${\Bbb H}^n$ as the warped product 
${\Bbb R}\times_{e^t} {\Bbb R}^{n-1}$, to be the nicest model of it
(here ${\Bbb R}$ and ${\Bbb R}^{n-1}$ are Euclidean).  One amuzing fact coming from 
the theory of geodesics in warped products, is how geodesics  of the hyperbolic plane 
are related to solutions of mechanical systems $x^{\prime \prime} = ce^{-x}$
($c$ is a constant) (see \S \ref{mechanics.surface}). Of course the interest here is not
 to analytically
 solve this equation, but rather 
to see how it can be solved geometricaly.

\begin{remark} [Generalization]
The situation of more general Riemannian symmetric space is more subtle. It involves
``multi-warped products''.
 This means that we have 
 $(L, h)$, and $(N, g)$, endowed with $T_1, \ldots T_k$  supplementary subbundles of 
$TN  $ ($= T_1 \oplus\ldots  \oplus T_k$), with restriction of the metric denoted
 $g_i$. We also have 
warping functions $w_1, \ldots, w_k $ defined on $L$, and construct from all, the metric  
$w_1g_1 \oplus \ldots \oplus w_kg_k.$  All Riemannian symmetric spaces
(e.g. $\SL(n, {\Bbb R})/\SO(n)$) admit such a representation. The geometry 
of such ``multi-warped products'' is quite delicate, at least more than the 
somewhat usual definition in the literature, where the plane fields $T_i$ are assumed
 to be integrable. However, 
it is the non-integrable case that covers the case of symmetric spaces. We think it is 
worthwhile   investigating this generalization.

\end{remark}

\section{Local warped products}   
\label{local.product}

A pseudo-Riemannian manifold which is 
a  warped product is in particular a global topological 
product. This is so restrictive (for instance  for physical applications)
 and 
we are led to localize the notion of warped products as follows.

\begin{definition}
Let $M$ be a pseudo-Riemannian manifold. 
A local warped product structure  
  on $M$
is a pair $({\mathcal L}, {\mathcal N})$ of  transversal foliations,
 such that the metric 
on adapted flow-boxes is a warped product. 
More 
precisely, for any point of $M$ there is
a neighborhood $U$, and 
a warped product pseudo-Riemannian manifold $L \times_wN$, 
 and an isometry 
$\phi: U \to  L \times_wN$, sending 
the foliation  ${\mathcal L}$
(resp. ${\mathcal N}$) to the foliation 
of $L \times N$ determined by 
the factor $L$ (resp. $N$).

A local warped 
product is called a local GRW structure if the factor ${\mathcal N}$
 has a constant curvature (i.e. each leaf of ${\mathcal N}$
is a pseudo-Riemannian manifold of constant curvature).
\end{definition}

\subsection{Geometry of submanifolds.} 
In the sequel, we will investigate conditions on a pair of foliations 
$({\mathcal L}, {\mathcal N})$ in order to  determine a local 
warped structure.
For this, let  $M$ be a pseudo-Riemannian 
  manifold, and 
  $S $ a non-degenerate 
submanifold of $M$, that is the metric restricted to 
$T_xS$ is non-degenerate for any $x \in S$.
Recall that the {\bf second fundamental form}, also 
called the {\bf shape tensor}, at $x$  is a bilinear 
map: $II_x: T_xS \times T_x S \to N_x$, where 
$N_x$ is the normal space of $T_xS$, which measures
how $S$ is far from being geodesic
($II$ is
well defined because of the non-degeneracy  hypothesis).

The submanifold $S$ is {\bf umbilic}  if for any 
$x\in S$, 
 $II_x$ 
 has the form 
$II_x = \langle,\rangle n_x$, where $n_x$ is  some  normal vector to 
$T_xS$.  In this case, the vector field
(along $S$) $x \to  n_x$  is  called the {\bf shape}
 vector field.
(the terminology force field is also pertinent
as  may  be seen from  Theorem   \ref{Maupertuis}).

The (totally) {\bf geodesic submanifolds}  correspond to
the case 
$n_x= 0$,  for all $x \in S$. 

We will also need the 
following notion:  $S$ is said to be 
{\bf spherical}, if it is umbilic, and 
furthermore, the shape vector field $x \in S \to n_x$, 
is parallel (along $S$).

When we consider umbilic submanifolds, we will 
always assume
that they have dimension 
$>1$. Indeed, every 1-dimensional submanifold  
is umbilic (but need not  to be spherical).

$\bullet$ Let us recall the 
geodesic invariance characteristic property of geodesic 
submaniflods. 
Let $x \in S$,
$u \in T_xS$, and let  $\gamma : ]-\epsilon , +\epsilon[
\to M$ be the geodesic in 
$M$
determined by $u$. If  $S$ is  geodesic, then the image
of $\gamma$
is contained
in $S$, for $\epsilon$ 
sufficiently small. 
This fact is  true also when 
$S$ is umbilic, if in addition 
$u$ is {\it  isotropic}. This is   a remarkable
rigidity fact in pseudo-Riemannian  geometry, which
has no counterpart   in 
Riemannian geometry.

\begin{example} 
{\em 
Take $M$ to be 
the pseudo-Euclidean  space of type $(p,q)$, i.e. ${\Bbb R}^{p+q}$
endowed with a pseudo-Euclidean form $Q$ of 
type $(p,q)$. 

A (connected) geodesic hypersurface
is an open set of  an  affine hyperplane. 
The (connected)  umbilic 
hypersurfaces are contained in hyperquardrics 
$Q(x-O) = c$, where $O \in {\Bbb R}^{p,q}$ and $c$ is
a constant (the proof is 
formally the same as in the 
Euclidean case).
 One can verify  that such a
hyperquadric is  ruled, that is, it  contains the isotropic 
lines which are somewhere
tangent to it.

In general, an umbilic submanifold is the intersection of 
a hyperquadric   with an affine plane
of some dimension.

In particular, one sees in the case 
of pseudo-Euclidean spaces, that umbilic 
submanifolds are spherical. This 
is true for all pseudo-Riemannian 
manifolds of constant curvature, but 
  not true in the general case.
}
\end{example}

\subsection{Tangential geometry of foliations.}
\label{tangential}

(See for instance \cite{B-H, Rov, Ton} for more details).  
 A foliation ${\mathcal F}$  is called geodesic,  umbilic or  spherical,
 if   its leaves
are geodesic,   umbilic or spherical, respectively.

Let $X$ be a vector field defined on an open subset $U \subset 
M$. 
We say that $X$ is
a  (${\mathcal F}$-) {\bf normal foliated} vector field, if   
  $X$ is orthogonal to ${\mathcal F}$, 
and
its  local flow $\phi^t$ preserves ${\mathcal F}$, i.e. it 
sends a leaf of ${\mathcal F}$ to a leaf of 
${\mathcal F}$ (everything  is restricted to $U$).

As in the case of an umbilic submanifold, 
 an umbilic foliation  ${\mathcal F}$
has a  {\bf shape } vector field 
$\overrightarrow{n}$    
defined by the relation
 $II = \langle,\rangle  \overrightarrow{n}$, 
where $II$ is the shape tensor.

\begin{lemma} 
\label{Lie.derivative}

Let ${\mathcal F}$ be a non-degenerate foliation 
of a pseudo-Riemannian 
manifold $(M, \langle,\rangle ) $. Let $f$ denote the first 
fundamental form of ${\mathcal F}$, that is 
the tensor which vanishes on $T{\mathcal F}^\perp$ and 
equals 
$\langle,\rangle $ on $T{\mathcal F}$, and denote by $II: T{\mathcal F} 
\times T{\mathcal F} \to T{\mathcal F}^\perp$ the second fundamental form.

Let $X$ be a normal foliated vector field, then 
the Lie derivative $L_Xf$ satisfies:
 $$(L_Xf) (u,v) = -2\langle II(u,v), X\rangle , $$
for all $u, v \in T{\mathcal F}$. (In other words, if
$\phi^t$ is the (local) flow of $X$, then, at any $x$, 
 $ \frac {\partial} {\partial t}(\phi^t_*f )_x \vert_{t=0}= 
-2\langle II_x(.,.), X\rangle$).

\end{lemma}

\begin{proof} Let $u$ and $v$ be two vector fields 
tangent to ${\mathcal F}$ which commute with $X$. Then 
by definition 
$(L_Xf)(u, v) = X.f(u,v)$, which also equals
$X. \langle u, v\rangle $. Now,
$ X.\langle u,v\rangle = 
\langle \nabla_Xu, v\rangle + 
 \langle u, \nabla_Xv\rangle $. By commutation, this becomes 
$\langle \nabla_uX, v\rangle +  \langle u, \nabla_vX\rangle $. 
Since 
$\langle X, v\rangle = \langle u, X\rangle  = 0$,  
 $(L_Xf)(u, v)=-\langle X, \nabla_u v\rangle -\langle X, \nabla_v u\rangle $,
 and
so by definition of $II$, we have:
$(L_Xf)(u, v) = -2\langle II(u,v), X\rangle $

\end{proof}

\begin{corollary} 
\label{Lie.flow}
If ${\mathcal F}$ is geodesic (resp. umbilic) then 
the flow of $X$ maps isometrically (resp. conformally) 
a leaf of ${\mathcal F}$ 
onto  a leaf of ${\mathcal F}$.

Conversely, if the flow of any normal foliated 
 vector field    maps
 isometrically (resp.
 conformally) leaves of ${\mathcal F}$ to leaves of ${\mathcal F}$, 
then ${\mathcal F}$is geodesic (resp. umbilic).

\end{corollary}

\begin{proof} 
The proof is  just the translation,
 with the above notation, 
of the fact that 
the flow $\phi^t$ maps 
 isometrically (resp.
 conformally) leaves of ${\mathcal F}$ to leaves of ${\mathcal F}$, 
into the equation:  $\phi^t_*f = f$ (resp. 
$\phi^t_*f= af$ for some scalar function $a$).

\end{proof}

Note however,  that there is no a such  characterization 
for spherical foliations. For example,  any
 (local)  umbilic 
foliation of the Euclidean space is   
 spherical, as it is just a foliation by round spheres. 
The flow of  a normal foliated 
vector field  maps conformally a sphere to
a sphere, but not more,
  for example not necessarily 
homothetically.

\section{Characterization of local warped products}

The following theorem is due  to S. Hiepko, but with a different  proof, and
especially with a purely ``analytic'' formulation. 
We said in a previous version of this article, that this analytic formulation could
explain why the article of Hiepko  
  \cite{Hie} seems to be not 
sufficiently known in the literature. Afterwards, we discover the work \cite{P-R}
by R. Ponge and H. Reckziegel, which contains a geometric approach.

\begin{theorem}
\label{characterization}
Let $(M, \langle,\rangle)$ be a pseudo-Riemannian manifold endowed
with a pair $({\mathcal L}, {\mathcal N})$ of non-degenerate foliations. 
This determines a local warped product structure with 
${\mathcal N}$ as a normal factor, iff, the foliations are orthogonal, 
${\mathcal L}$ is geodesic, and ${\mathcal N}$
is {\em spherical}.
\end{theorem}

\begin{proof}
Let ${\mathcal L}$ and ${\mathcal N}$ be two 
orthogonal foliations. Locally, at a topological level, 
we may suppose that $M= L \times N$, and that the foliations
${\mathcal L}$ and ${\mathcal N}$ correspond to those  
determined by the factors $L$ and $N$. 
Let $(x, y)$ be a fixed  point in  $L \times N$. The metric 
on $M$ at $(x, y)$ 
has the form $h_{(x,y)} \bigoplus f_{(x, y)}$, where 
 $h_{(x,y)}$ (resp.  $f_{(x, y)}$) is a metric 
on $L \times \{y\}$ (resp. on $\{x\} \times
N$). Note that a normal foliated 
vector field for  ${\mathcal L}$ is just a 
vector field of the form 
$X(x, y)= (0, \bar {X}(y))$, where $\bar{X}$
is a vector field on $N$, and similarly 
for ${\mathcal N}$.

By   Corollary  \ref{Lie.flow},   ${\mathcal L}$
is geodesic,  iff 
$h_{(x,y)} = h_y$.  In the same way, ${\mathcal N}$
is umbilic, iff there is a 
function 
$w(x, y)$ such that 
$f_{(x, y)} = w(x, y) f_x$. Therefore, the fact that ${\mathcal L}$
is geodesic and ${\mathcal N} $ is
umbilic, is equivalent to that 
the metric $\langle, \rangle$  of $M$ is 
a {\bf twisted product}
$h \bigoplus w g$, where $h$ and $g$ are metrics on 
$L$ and $N$ respectively, and 
$w$ is a function on $L \times N$.

By choosing a point $(x_0, y_0)$, we may suppose 
that $g = f_{(x_0, y_0)}$, and hence $w(x_0, y)
= 1$, for all $y \in N$.

The fact that this  metric is a warped product  
means exactly that $w$ is a function of $x$ alone. 
Therefore, 
the statement of the theorem reduces  now
to the  equivalence between the two facts, 
 $w$  being  constant along
${\mathcal N}$,  and  ${\mathcal N}$
 being  spherical. 

To check  this,  let $\bar{X}$ and 
$\bar{Y}$ be two vector fields on 
$L$ and $N$,  respectively, and let 
$X$ and $Y$ be the corresponding vector 
fields on $M$, which are 
normal foliated 
relatively 
to ${\mathcal N}$ and ${\mathcal L}$, respectively.

Since ${\mathcal N}$ is umbilic, 
   $II=  f \overrightarrow{n}$, where 
$f$ and $II$ are the first and second fundamental forms for 
${\mathcal N}$ respectively, and $\overrightarrow{n}$ 
is its shape vector field.

We have, 
$Y\langle \overrightarrow{n}, X\rangle  = 
\langle \nabla_Y\overrightarrow{n}, X\rangle  + \langle \overrightarrow{n}, \nabla_YX\rangle 
= \langle \nabla_Y\overrightarrow{n}, X\rangle  + \langle \overrightarrow{n}, 
\nabla_XY\rangle $, since $X$ and $Y$ commute. 

Since ${\mathcal L}$ is geodesic, 
$\nabla_XY$ is orthogonal to ${\mathcal L}$, 
in particular, 
$\langle \overrightarrow{n}, 
\nabla_XY\rangle = 0$. It then follows that 
$Y\langle \overrightarrow{n}, X\rangle  = 
\langle \nabla_Y\overrightarrow{n}, X\rangle $.

Lemma \ref{Lie.derivative} says that $X.w = -2 \langle \overrightarrow{n},X\rangle $, 
and hence
$ Y.(X.w) = \langle \nabla_Y\overrightarrow{n},X\rangle $. By definition,  
${\mathcal N}$ is spherical iff $\langle \nabla_Y\overrightarrow{n},X\rangle 
=0$, for all $X$ and $Y$, 
which is thus equivalent to 
  $Y.(X.w)= 0$. This last equality, applied to 
a fixed $Y$, and an arbitrary $X$, 
 means that
$Y.w $ is a function of $y$ only, say 
$Y.w = a(y)$. 
But, since $w(x_0, y) = 1$, it follows that 
$Y.w = 0$. Applying this to an arbitrary
$Y$, leads to the fact  that $w$
does not depend on  $y$, which in turn  means 
that the metric is a 
warped product. 
\end{proof}

\section{Transverse geometry of foliations}

\label{transverse}

Theorem \ref{characterization} is expressed by means of  
tangential
properties of foliations, i.e.  by those  of individual leaves. 
 Sometimes, it is  also interesting to consider 
the transverse structure of these foliations, i.e.
the properties of their  holonomy maps
(see for instance \cite{Mol} as a reference about   such notions).
These holonomy maps are 
  especially easy to realize, 
for a foliation $\mathcal F$, 
when 
 the orthogonal $T{\mathcal F}^\perp$
is
integrable, that is, when it  determines a foliation 
 say  ${\mathcal F}^\perp$.
The holonomy maps  of ${\mathcal F}$ are 
thus just the local diffeomorphisms between 
leaves of ${\mathcal F}^\perp$, obtained by integrating 
${\mathcal F}^\perp$- normal foliated
  vector fields (see \S \ref{tangential} for their definition).

The foliation  ${\mathcal F}$ is said to be 
{\it transversally pseudo-Riemannian}  
if its holonomy preserves the pseudo-Riemannian 
metric on $T{\mathcal F}^\perp$. Similarly 
one defines the fact that ${\mathcal F}$
is transversally conformal
(resp. 
transversally homothetic).
Using this language, 
 the previous developments  imply 
 straightforwardly the following fact.

\begin{fact}
\label{transverse.critere}
A pair $({\mathcal L}, {\mathcal N})$ determines
a local warped product structure, iff 
${\mathcal L}$ is transversally homothetic and
${\mathcal N}$
is transversally 
pseudo-Riemannian.

\end{fact}

In general (i.e. in a not 
necessarily warped
product situation), 
we have the following duality 
between tangential and transverse structures of foliations.

\begin{fact} 
\label{geodesic.critere}
Let $\mathcal F$ be a foliation  admitting
an orthogonal  foliation ${\mathcal F}^\perp$.
Then ${\mathcal F}$ is geodesic (resp. umbilic)
iff 
${\mathcal F}^\perp$ is transversally pseudo-Riemannian
 (resp. conformal),  
that is more precisely, 
  the holonomy maps of the foliation ${\mathcal F}^\perp$, 
seen as local diffeomorphisms between leaves of ${\mathcal F}$, 
preserve the metric (resp. the conformal structure)
induced on these leaves (of ${\mathcal F}$).
\end{fact}

\subsection{Proof of Fact \protect{\ref{geodesic.extension}}.}

Let  $S $ be  a submanifold 
of $N$, and   $M = L \times _w N$. In order to prove the equivalence, 
$S$ a geodesic submanifold  in $N$ $ \Longleftrightarrow$   $ L \times S$ a geodesic submanifold  in $M$, 
it suffices to consider the case where the dimension of $S$
is 1, i.e. $S$ a (non-parameterized) geodesic (curve). Indeed the general case reduces to 
the   1-dimensional one by considering geodesic (curves) of $S$.

To simplify let us  suppose that $N$ is Riemannian, the general
case needs only  more notations. 

A geodesic such as $S$ can be locally extended 
to a 1-dimensional foliation ${\mathcal F}$ with an orthogonal foliation 
${\mathcal F}^\perp$. To see this, take a hypersurface 
$S^\perp \subset  N $ which is somewhere orthogonal to $S$, then 
the leaves of ${\mathcal F}^\perp$
are the parallel hypersurfaces of $S^\perp$. More precisely, they are 
the levels of the distance function $x \to a(x)= d(x, S^\perp)$. The 
leaves of ${\mathcal F}$ are the trajectories of 
$\nabla a$, the gradient of $a$. Thus ${\mathcal F}^\perp$
is a transversally pseudo-Riemannian foliation of $N$. By taking the
product of
the leaves of ${\mathcal F} $ with $L$, one may define $L \times {\mathcal F}$  
as a foliation of $M$. The orthogonal foliation
 $(L \times {\mathcal F})^\perp$ of 
$L \times {\mathcal F}$ is naturally identified 
with ${\mathcal F}^\perp$ (the leaf of 
$(x, y) \in L \times N$ is 
$\{x\} \times {\mathcal F}^\perp_y$).
From the form of the warped product
metric, one sees that, like 
 ${\mathcal F}^\perp$, $(L \times {\mathcal F})^\perp$
is a transversally pseudo-Riemannian foliation.
Therefore, $L \times {\mathcal F}$ is a geodesic foliation, and 
in particular $L \times S$ is geodesic in $M$.

The implication,  
 $ L \times S$ geodesic in $M
 \Longrightarrow S$ geodesic in $N$, is in fact 
 easier than  its  converse that we have just proved. Indeed, if $\nabla$ is the connection
on $N$, and $X$, $Y$ are vector fields tangent to 
$S$, then $\nabla_XY$ is tangent to
$L \times S$ (since it is geodesic), and hence its
orthogonal  projection on $N$ 
is tangent to $S$, that is,  
$S$ is geodesic in $N$.

\begin{remark} Although, we are not interested here in  global aspects, let us mention
that there are many works  about  the structure of geodesic, umbilic,  transversally  Riemannian, transversally conformal foliations on compact manifolds. As an example, we may quote   the references \cite{Bru-Ghys, Car-Ghys,  Mol}.

\end{remark}

\section{Isometric actions of Lie groups}

(Local) 
isometric actions of Lie groups on pseudo-Riemannian 
manifolds 
generally give rise to a warped product structure.  
In some sense, this phenomenon is the converse of the 
trivial isometric extension Fact \ref{isometry.extension}.
The following  statement 
may be used to 
settle a warped product structure in many situations.
It 
 unifies and generalizes most of the existing results 
on the subject (see  for instance
 \cite{G-K, Har, Kar, Lap, Mat, Pod}).

\begin{theorem} 
\label{action.criterion.1}
Let $G$ be a Lie group acting (locally) 
isometrically on a pseudo-Riemannian manifold $M$. Suppose that 
the orbits have a constant dimension and
thus determine a foliation ${\mathcal N}$. 

Suppose further that the leaves of ${\mathcal N}$ are non-degenerate, and 
that the isotropy group in $G$ of any $x \in M$, 
acts {\em absolutely irreducibly} on 
$T_x {\mathcal N}$, i.e. its complexified representation  is
irreducible. 

Suppose
finally that  
 the orthogonal of ${\mathcal N} $ is integrable, say it is tangent to 
a foliation 
${\mathcal L}$. 
Then $({\mathcal L}, {\mathcal N})$ 
determines a local warped product structure, with 
${\mathcal N}$
as a normal factor.

\end{theorem}

\begin{proof}
The question is local, and so  we can 
suppose  
the situation is topologically 
trivial. For two nearby leaves
$ N_1$ and $ N_2$, there is 
a projection $p:  N_1 \to  N_2$, defined by: \;  
$p(x)$  is the unique point of the intersection 
of  ${\mathcal L}_x $ ($ ={\mathcal N}^\perp_x $) with $ N_2$ (for $x \in N_1$).
This projection commutes with the action of $G$.
The pull back by $p$ of the metric  
 on $T_yN_2$
(at $y = p(x)$) is 
another  metric on $T_xN_1$, invariant 
by $G_x$. The fact that $G_x$ is absolutely irreducible
just implies that the two metrics are proportional.
Therefore $p$ is conformal. But since $p $
commutes with the (transitive) $G$-action on $N_1$ and $N_2$, 
$p$ must be homothetic.

It is easy to relate the projection 
$p$ 
to the transverse holonomy of ${\mathcal L}$ (as developed in 
 \S \ref{transverse}), proving   that 
${\mathcal L}$ is transversally homothetic. It is 
equally straightforward to relate the transverse 
holonomy of ${\mathcal N}$
to the $G$-action, and  deducing   that ${\mathcal N}$
is transversally pseudo-Riemannian, and therefore
$({\mathcal L}, {\mathcal N})$ determines a local warped
product structure (by Fact \ref{transverse.critere}).

\end{proof}

\begin{example} {\em  The absolute irreducibility 
hypothesis cannot be relaxed to an ordinary irreducibility one.
To see this let  $N$ be a Lie group, and let $G$ be the product 
$N \times N$ acting on $N$ by the left and the right, that is
$(\gamma_1, \gamma_2)x = \gamma_1^{-1}x\gamma_2$. 
The isotropy group of this action at 
the point $1$, is nothing but the adjoint action of 
$N$ on itself. It  is irreducible (resp. 
absolutely irreducible) iff $N$ is a simple 
(resp. an absolutely simple) Lie group (by definition).
In the case  $N$ is  simple but non absolutely simple,
e.g. $\SL(2, {\Bbb C})$,  the isotropy action  preserves 
exactly (up to linear combination) two non-degenerate quadratic forms, 
those given by the real  and the imaginary parts of the 
Killing form of $N$, seen as a complex group. This 
gives  
    two $G$-invariant non-proportional 
metrics 
$\alpha$ and $\beta$ on $N$.

Let  $(L, h)$ be  another 
pseudo-Riemannian manifold, and let $ f: L \to 
{\Bbb R}$ be a real function. Endow 
$L \times N$ with the metric $h \bigoplus (f\alpha + \beta)$.
This is not  a warped product.
}
\end{example}

The following result describes an example of a situation
where the hypotheses of Theorem  
\ref{action.criterion.1} are satisfied.

\begin{theorem}
\label{action.criterion.2}
Let $G$ be a Lie group acting (locally) 
isometrically on a pseudo-Riemannian manifold $M$. Suppose that 
the orbits are non-degenerate  having  a constant dimension and
so  determine a foliation ${\mathcal N}$. 

Suppose that 
 the isotropy group in $G$ of any $x \in M$, 
acts absolutely  irreducibly on 
$T_x {\mathcal N}$,  
and that  the metric on  
 the orthogonal of ${\mathcal N}$ is 
 definite (positive or negative), 
and in opposite the metric 
on ${\mathcal N}$ is non-definite. 
Then,  the orthogonal of ${\mathcal N} $ is
 integrable, and  the action determines a local
warped product.

\end{theorem}

\begin{proof} The warped product structure will follow
from Theorem 
\ref{action.criterion.1} once  we show that 
the orthogonal of ${\mathcal N}$ is integrable. We will in fact  prove this 
integrability,  
under the hypothesis  that the isotropy is irreducible (not 
necessarily 
absolutely irreducible).
Consider $\alpha: 
T {\mathcal N}^\perp  \times T{\mathcal N}^\perp
\to T{\mathcal N}$ the  bilinear form 
(obstruction to the integrability 
of $T{\mathcal N}^\perp$) 
$\alpha (u, v)= $ the projection 
on 
$T{\mathcal N}$ of the bracket 
$[u, v]$, where $u$ and $v$ are 
 vector fields on $M$ with values in  $ T{\mathcal N}^\perp$.
  Let $x \in M$, and consider the 
subset $A_x$ of $T_x{\mathcal N}$ 
which consists of the elements
$\alpha(u, v)$, for 
$u$ and $v$ of length $\leq 1$. This set is 
compact, and is invariant by the isotropy 
group 
$G_x$. Since $\alpha $
is equivariant, $G_x$ acts 
precompactly on $A_x$ since it acts so 
on $T_x{\mathcal N}^\perp$. It then follows that 
$G_x$
acts precompactly on the linear space $B_x$ generated by 
$A_x$. If $A_x = 0$, $\alpha= 0$, and we are done, if not
$B_x = T_x{\mathcal N}$
by irreducibility.  Thus, $G_x$   preserves a positive scalar 
product 
on 
$T_x{\mathcal N}$. 
 But, it also preserves the initial non-definite 
pseudo-Riemannian product. Polarize this latter with respect to
the invariant positive 
scalar product, we get a diagonalizable endomorphism, that has
both positive and negative eigenvalues since 
the pseudo-Riemannian product 
is non-definite. This contradicts the irreducibility.

\end{proof}

A similar  argument yields the following useful fact.

\begin{fact}
\label{dimension4}
 Let $\SO(3)$ act isometrically on a $4$-Lorentz manifold
with $2$-dimensional orbits. 
Then, this determines
a local warped product structure, with a local model
$L \times_w S^2$ or $L \times_w {\Bbb R}P^2$. (One may 
exclude the projective plane case by a suitable orientability
hypothesis).

\end{fact}

\section{Geodesics. Maupertuis' principle}

The goal here is to understand 
the geodesics of 
 a warped product $M = L \times_wN $. 
Let $\gamma(t) = (x(t), y(t))$ be 
such a geodesic.  

Fact 
\ref{geodesic.extension} implies that  $y(t)$
 is a (non-parametrized)
geodesic in $N$. To see this, let $S$ be a 
(1-dimensional) geodesic
of $N$ such that $\gamma(t) $ is somewhere 
tangent to $L \times S$.  Fact 
\ref{geodesic.extension} says that $L \times S$ is geodesic 
in $M$, and therefore contains the whole of 
$\gamma(t)$,  which thus    projects onto an open 
subset of the geodesic $S$.

Now, it remains   to draw equations, and especially to interpret them,
for $x(t)$, and also    determine the  parameterization 
of $y(t)$. Here, the idea 
is to replace $M$ by 
$L \times_w S$, which transforms the problem to
a simpler one, that is the 
 case  where $N$ has dimension $1$ (since $L \times_w S$ is geodesic in $M$, we do
 not need the rest of $M$ to understand  a geodesic contained in $L\times_w S$!).

\subsubsection*{Clairaut first integral.}
 
The previous discussion 
allows one 
to restrict the study to 
warped products of the  type
$L \times_w ({\Bbb R}, c_0dy^2)$, where $y$ denotes 
the canonical coordinate on ${\Bbb R}$, and 
$c_0$ is $-1$, $+1$ or 0. 
 Of course, the case $c_0 =0$, i.e.  when the 
non-parameterized geodesic $y(t)$ is lightlike, 
 does not
really  correspond to 
a pseudo-Riemannian structure, so, let us assume
 $c_0 \neq 0$.

Actually, the geodesic $S$ above in not 
necessarily complete, that is, 
it is not parameterized by ${\Bbb R}$ but just by an open subset
of it. However, our discussion here is local in nature, so to 
simplify notation,
we will assume   $S$  complete.

The isometric action of (the group) ${\Bbb R}$ on  
$ ({\Bbb R}, c_0dy^2)$
extends to an isometric flow on 
$L \times_w ({\Bbb R}, c_0dy^2)$ (by  
Fact \ref{isometry.extension}). 

The so called 
Clairaut first integral 
(for  
the geodesic flow on the tangent bundle of 
$L \times_w ({\Bbb R}, c_0dy^2)$) means here that 
$\langle y^\prime (t), \partial
 /\partial y \rangle  $
is constant, say, it equals $c_1$ (remember $\gamma(t) = (x(t), y(t)$ is our geodesic).
Since $y^\prime(t) $ and $\partial/ \partial y$
are collinear, 
it  follows that: 
$$\langle y^\prime(t), y^\prime(t)\rangle  = 
\frac{\langle y(t)^\prime, \partial /\partial y \rangle ^2} 
{\langle \partial /
\partial y, \partial /\partial y \rangle } =(\frac{c_1^2}{c_0}) \frac{1} {w(x(t))}$$

In dimension 2, that is, 
dim$L =1$,  the Clairaut integral together with
the energy integral: $\langle \gamma^\prime(t), \gamma^\prime(t) \rangle
 = $ constant,  suffice to understand completely
the geodesics. The remaining developments concern the case
dim$N \geq 2$.

\subsubsection*{The shape vector field.}

The distortion of the warped product  structure, i.e. the 
obstruction to  being  a direct product is encoded in 
$\nabla w$, the gradient of $w$ (with respect 
to the metric of $L$).  

Obviously, the fact that the foliation
${\mathcal N}$ 
(i.e. that with leaves 
$\{x\} \times N$)
be geodesic is also an obstruction for  
the warped  product  to be direct. The following 
fact is a quantitative version of this obstruction.

\begin{fact}
\label{shape.vector}
 The shape vector field 
$\overrightarrow{n}$  of  ${\mathcal N}$ 
  is a
${\mathcal N}$-foliated vector field.  More exactly 
(identifying  $TM $  with $ TL \times TN$):
$$\overrightarrow{n} (x,y) =
 \frac{-1}{2}(\frac{\nabla w (x) }{w}, 0)$$

\end{fact}

\begin{proof}
 With the notations of Lemma \ref{Lie.derivative},
we have $f= wg$, and thus (by  definition of
$\overrightarrow{n}$)
$L_Xwg = -2  \overrightarrow{n} f$, and 
on the other hand
 $L_X(wg) = (X.w)g = \frac{X.w} {w} wg$.

\end{proof}

(Observe that $\frac{\nabla w  }{w}$ is well defined even 
for local warped product structures).

\subsubsection*{Projection onto $L$.}
We consider the case where $x(t)$ and 
$y(t)$ are regular curves, since the question is local and the 
other cases are quite easier. Therefore,  these curves
 generate a surface,
$(x(t), y(s))$, whose tangent bundle is generated
 by the natural 
frame $(X, Y)$. Since 
$X$ and $Y$ commute, we  have $\nabla_{X+Y} (X+Y)= \nabla_XX + 
2\nabla_XY + \nabla_YY$.

Since ${\mathcal L} $ is geodesic, $\nabla_XY$ is tangent to
${\mathcal N}$ (indeed, if $Z$ is tangent to ${\mathcal L}$, then, 
$\langle \nabla_XY, Z \rangle = 
X \langle Y, Z \rangle + \langle Y, \nabla_X Z \rangle 
=0$,   because ${\mathcal L} $ is geodesic). 
By definition, the projection of 
$\nabla_YY$ on 
$T{\mathcal L}$ equals $\langle Y,Y\rangle  \overrightarrow{n}$. 
From 
this and Fact \ref{shape.vector}, we deduce that 
the projection of $\nabla_{X+Y}(X+Y)$ on $T{\mathcal L}$ 
equals  $\nabla_XX -(1/2) \langle Y,Y\rangle  
\frac{\nabla w}{w}$, which
 must vanish 
in the geodesic case. Replacing $\langle Y, Y\rangle  $
 ($= \langle y^\prime(t), y^\prime(t)\rangle$)
by its expression above, we obtain:
$$\nabla_XX= 
( \frac{1}{2} \frac{c_1^2}{c_0}) \frac{1} {w(x(t))} 
\frac{ \nabla w}{w}
= ( - \frac{1}{2} \frac{c_1^2}{c_0})   
 (\nabla \frac{1}{ w}) (x(t))  $$
This proves the following.

\begin{fact}
\label{projection}
The projections onto  $L$ of the geodesics 
of $L \times_w ({\Bbb R}, c_0dy^2)$ are exactly the trajectories
 of the mechanical systems 
on $L$ with potentials $\frac{c}{w}$, i.e. curves on $L$ 
satisfying an equation of the form: 
$$x^{\prime \prime} = - \nabla (\frac{c}{w}) (x)$$ where 
$c$ runs over ${\Bbb R}^+$ (resp. ${\Bbb R}^-$) if   
$c_0 >0$ (resp. if $c_0 <0$).

\end{fact}

From this, we deduce the following 
 fact for a  general $N$.

\begin{theorem} 
\label{Maupertuis}
[Maupertuis' principle] 
The projections onto  $L$ of the geodesics 
of $M = L \times_w N$ are exactly the trajectories of
 the mechanical systems 
on $L$ with potentials $\frac{c}{w}$,  for 
$c \in {\Bbb R}$ if the metric on $N$ is 
non-definite, 
for $c  \in {\Bbb R}^+$ if 
the metric on $N$ is positive definite, and 
for $c \in {\Bbb R}^-$ if the metric on 
$N$ is negative definite.

\end{theorem}

\subsubsection*{Equations.}

In the case where $y(t)$ is not lightlike, its
parameterization  is  
 fully determined, whenever   $x(t)$ is known, by using the first integral
$\langle y^\prime(t), \partial / \partial y \rangle  = c_1$. 
Indeed   one can  identify $y^\prime(t)$ with 
$y^\prime(t) \partial / \partial y$, and 
thus with help 
of the notation above, 
 $y^\prime(t) = \frac {c_1} { \langle \partial / \partial y, \partial
/ \partial y\rangle (x(t))}= \frac{c_1} {c_0 w(x(t))}$.

There is no analogous equation
in the case
where $y(t)$ is lightlike. 
 Let us derive the general equation in another way
which covers the lightlike case. 
 From 
the calculation before Fact \ref{projection}, we have
$\nabla_YY + 2\nabla_YX = 0$. Now, for all $Z$ tangent to 
${\mathcal N}$, 
$\langle \nabla_YX, Z\rangle  = -\langle X, \nabla_YZ\rangle  = 
-\langle X, \overrightarrow{n}\rangle  \langle Y, Z\rangle $.
Therefore, $\nabla_YX = -\langle X, \overrightarrow{n}\rangle  Y=
 \frac{dw(X)}{2w} Y$, which proves:

\begin{fact} The curve $y(t)$ has a geodesic support, and
its parameterization is coupled with the 
companion curve $x(t)$ by means of the equation:
$$ y^{\prime \prime} = 
-\frac{\partial} { \partial t} (\log w) (x(t)) y^\prime$$

\end{fact}

\subsubsection*{Mechanics on $M$.} 
The previous discussion relates
the geodesics of $M$ to trajectories of mechanical 
systems on $L$. 
Let us  now start  with
a mechanical system
$\gamma^{\prime \prime} = - \nabla V$ on $M = L \times _wN$
itself.
 We assume that the 
potential $V$ is constant on the leaves $\{x\} \times N$, 
and thus may be   identified with  a function 
  on $L$.   Essentially, by  the same arguments, one proves.

\begin{proposition}
\label{mechanics}
  Consider on 
$M = L \times_wN$, the equation 
$\gamma^{\prime \prime} = - \nabla V$, where $V$ is a function on $L$.
Then, 
 the projections of its  trajectories on $N$ 
are non-parameterized geodesics of $N$, and their projections  
on $L$ are trajectories of mechanical systems 
on $L$ with potentials of the form
$V + \frac{c}{w}$, where $c$ runs over ${\Bbb R}$
if the metric on 
$N$ is nondefinite, and otherwise, $c$
 has the same sign as the metric of $N$.

\end{proposition}

\begin{corollary}
\label{mechanics.surface}
  
 If $M = L \times_wN$, has dimension 2, 
i.e. $L$ and $N$ are locally isometric to ${\Bbb R}, \pm dt^2)$, then,  
  the trajectories of the  equation 
$\gamma^{\prime \prime} = - \nabla V$, where $V$ is a function on $L$ 
are completely  determined by means of:

i) their projection on $L $ satisfy $x^{\prime \prime}  = V + \frac{c}{w}$, and

ii) they satisfy the the energy conservation law
$\langle \gamma^\prime(t), \gamma^\prime(t) \rangle
 + V(\gamma(t)) = $ constant

\end{corollary}

\begin{example} { \em This applies in particular to solve the geodesic equation 
on the hyperbolic plane ${\Bbb H}^2 = {\Bbb R} \times_{e^t} {\Bbb R}$.

}
\end{example}

\section{Examples. Exact solutions}

In the sequel, we will in particular   consider 
examples of warped product structures on  exact solutions,
i.e. 
 explicit 4-Lorentz manifolds with  an   
explicit Einstein tensor (for details, one may for instance consult
\cite{H-E}, \cite{Mis},  \cite{O'n}...).
In fact,  warped products are omnipresent in cosmological 
models, because of their simplicity and symmetry advantages,
 as explained in the introduction. However, the most important 
use of warped product is in formulating  expanding universes. 
This needs  the warped product to be 
of ``physical'' type. Let us formulate  precisely what we 
mean by this.

\begin{definition} 
\label{anti.physical}
We say  that a warped product structure on a 
  Lorentz 
manifold  $M = L \times_wN$ is  {\bf physical} 
if the metrics on
both   the factors $L$ 
and  $N$ are definite (one positive and the other  negative).
Otherwise, the warped product structure 
is called {\bf  anti physical}.
The same definitions apply for local warped products and 
   GRW structures.

\end{definition}

Equivalently, the warped product 
is physical if $N$ is
spacelike or
locally isometric to
$({\Bbb R}, -dt^2)$. The dynamical counterpart in the first case, i.e. 
when 
$N$ is spacelike, is that of a universe in expansion (\S
\ref{dynamic}), and in opposite, 
a warped product structure for
which 
$N$ is 
locally isometric to
$({\Bbb R}, -dt^2)$, 
 corresponds to a static universe (\S \ref{static}).
 
The warped product $M = L \times_w N$
is anti physical iff one of the factors $L$
or 
$N$ is Lorentzian.

\subsection{Expanding universes:  classical 
Robertson-Walker spacetimes.}
\label{dynamic}
Here, $L$  is an interval $(I, -dt^2)$, 
and $N$ a 3-Riemannian manifold of constant curvature.  
Recall that the  
-energy tensor satisfies
(or say, it is defined by)   the Einstein 
equation: $T = (1/ 8 \pi) (Ric- \frac{1}{2} Rg)$
($Ric$ and $R$  are respectively the Ricci and
scalar curvature of $\langle, \rangle$). Here, it has 
the form of a perfect fluid:
$T = (\mu + p) \omega \bigotimes \omega + p \langle, \rangle$,
where $\omega=$  the dual 1-form of $\frac{\partial} { \partial t}$ 
(with respect to the metric $\langle,\rangle$)
and the  functions
$\mu$ (energy density) and $p$ (pressure), 
are determined 
by the warping function $w$ 
(by means of 
 the Einstein equation). In fact,  the condition that 
$N$ has a constant curvature is
exactly needed to
get a perfect fluid.

\subsection{Static universes.}
\label{static}
Not only expanding universes involve a
warped product structure, but also
the static ones, which are   defined as
those spacetimes having non-singular timelike Killing fields
with an integrable orthogonal distribution. The 
fact that this gives a local warped  structure with 
the trajectories of the given Killing
field as a normal foliation, is a special elementary 
case of Theorem \ref{action.criterion.1}

Conversely, by the  isometric extension
 Fact   \ref{isometry.extension},  
a warped product $M = L \times_wN$, with $N$ 
locally 
isometric to $({\Bbb R}, -dt^2)$ (essentially 
$N$ is an interval endowed with a negative metric) 
is static. 
Note however that a local warped product 
with a normal factor
locally 
isometric to $({\Bbb R}, -dt^2)$ is not necessarily
static, since there is an ambiguity in defining a global 
Killing field  as desired. The natural notion 
that can  be considered here is that of a locally static spacetime, which 
  will  thus be equivalent to having a local warped product structure 
with a normal factor 
locally isometric to $({\Bbb R}, -dt^2)$.

\subsubsection{A naive gravitational model.}
Consider the warped product $M = ({\Bbb R}^3,$ Euclidean)
 $\times_r ({\Bbb R},-dt^2)$,
where
$w= r: {\Bbb R}^3 \to{\Bbb  R}$
is the radius function. 
 (The warped product metric is non-degenerate only for $r \neq 0$, so
more exactly,  $M$ equals $({\Bbb R}^3-\{0\}) \times _r {\Bbb R}$).

From Theorem \ref{Maupertuis}, the projection of the geodesics of $M$ 
are the trajectories of the mechanical systems 
 on the Euclidean space ${\Bbb R}^3$, with 
  potentials of the form
$V= c/r$, where $c $
is a non-positive constant.
By this, one obtains in particular the trajectories of 
the Newtonian 
potential
$V=- 1/r$.  
 
In fact,  this process gives a (naive) relativistic static 
model
$ L \times_w ({\Bbb R},-dt^2)$ associated to
any
negative  potential $V = -1/w: L \to {\Bbb R}^-$
 on a Riemannian manifold $L$.

One flaw of such a  model is that it is   not characteristic of 
the 
initial potential $V$,  since it cannot distinguish between the potentials 
$cV$ for 
different (non-negative) constants $c$, and it
recovers in particular the geometry of $L$, 
for $c =0$. In fact,  except for exceptional cases, 
two 
warped products $L \times_w {\Bbb R}$ and $L \times_{cw} {\Bbb R}$ are 
isometric by means of the {\it unique}  mapping
$(x,t ) \to (x, ct)$, which acts  as a time dilation.
Therefore, the model would be specific 
of the potential if one introduces an extra structure  breaking  
time dilations.

It seems interesting to investigate some features of
these spaces, especially from the viewpoint of being perfect fluids.

``Newtonian spacetimes'' (see for instance \cite{Mis}, \S 12) were introduced by 
E. Cartan for the goal of making geodesic the dynamics under a mechanical system
derived from a potential. The structure there is that  
of an affine connection, which is poor, compared to the Lorentz
structure here. 

We think it is worthwhile  investigating a synthesis of  all the approaches 
to geodesibility processes of dynamical systems.

\subsection{Polar coordinates.} 
\label{polar.2}

The polar coordinates at 0 endow  the Minkowski space
$ ({\Bbb R}^{n,1}, \langle ,\rangle )$ with a warped product structure defined 
away from
 the light cone $\{x/ \langle x,x\rangle  =0\}$. Inside the cone, the structure 
is physical, with a normal factor homothetic to
the hyperbolic space ${\Bbb H}^n$, and outside the cone, the structure 
is anti-physical, with a normal factor homothetic to the de Sitter space 
$\{x / \langle x,x\rangle  = +1\}$.

\subsection{Spaces of constant curvature.}

(See for instance \cite{Wol} for 
some facts  on  this subject). The spaces of constant curvature are 
 already ``simple'', 
but one may need for some calculations to write them 
as (non-trivial) warped products, for instance polar coordinates 
on these spaces give rise to warped product structures defined 
on some open sets.

 Recall that 
 for these spaces, umbilic submanifolds (with dimension $\geq 2$) are
spherical, and also have  constant curvature. In particular,  
a warped product structure in this case is a GRW structure. 
(In dimension 4, and 
if the normal factor is 
spacelike,  one obtains 
  a classical Robertson-Walker structure, \S  \ref{dynamic}.
 The perfect fluid has in this case
constant density and pressure).

One can  prove the following fact which
 classifies the warped products  
in this 
setting.
(See \cite{San} for a study  of global  warped products
of physical type). 

\begin{fact} Let $N$ be an umbilic (non-degenerate) submanifold in a
  space of constant curvature $X$. Consider 
the foliation ${\mathcal L}$, defined on a neighborhood
${\mathcal O}(N)$  of $N$,
having as  leaves the 
 geodesic submanifolds orthogonal to $N$.

Then, the orthogonal distribution of ${\mathcal L}$ is integrable, 
say it is tangent to a foliation ${\mathcal N}$. 
Moreover, $({\mathcal L}, {\mathcal N})$ 
determines a GRW structure.

Furthermore,  ${\mathcal N}$ is the orbit foliation of the isometric 
action of 
  a natural subgroup  $T(N)$ of 
 Isom$(X)$ preserving $N$. In the case where $N$ is a geodesic submanifold, 
$T(N)$ is the group generated by the transvections along the geodesics of $N$.
(A transvection along a geodesic is an isometry 
which induces parallel translation along it).

\end{fact}

\subsection{Schwarzschild spacetime.}

The building
of Schwarzschild spacetime gives an excellent example 
of how 
various warped product structures
may be  involved. We will  essentially study it from this
point of view.
This spacetime models  a relativistic
 one body  universe (a star). Its 
construction is accomplished by translating the physical content into geometrical
 structures, and making at each stage ``necessary'' 
topological simplifying assumptions.

The spatial isotropy  around
the star leads to the first geometric structure, 
 formulated by the fact that $\SO(3)$ acts 
isometrically 
with 2-dimensional 
orbits.  From Fact \ref{dimension4}, 
we get a local warped product  of the type
$L \times_w S^2$ (one excludes the 
${\Bbb R}P^2$-case by an appropriate orientability extra  hypothesis).
One then makes  
the topological simplifying  hypothesis that  
the warped product is
 global.
 
This warped product (in particular the function $w$)
is canonical (it has a  physical 
meaning) and is in particular compatible with 
the additional structures.

The second geometrical hypothesis
on the spacetime is 
that it  is static
 (which in fact leads to another 
local warped product 
structure with a normal factor locally 
isometric to
$({\Bbb R}, -dt^2)$).

The compatibility between 
structures, 
  implies,  essentially, that  
the surface $L$ itself is static. Thus 
(after topological simplification)
$L$ is   a warped product 
$({\Bbb R}, g) \times_v ({\Bbb R}, -dt^2)$. 
(where $g$ is   some metric on ${\Bbb R}$).

By compatibility, the warping function $w$
is invariant by the Killing timelike field 
$\frac{\partial} {\partial t}$
on $L$. Its gradient is thus tangent to 
the first factor ${\Bbb R}$ of $L$.  
Another topological simplification 
consists in 
assuming    
that $w$ is regular, 
namely,  $r= \sqrt{w}$
is a global coordinate function on ${\Bbb R}$ (the first factor of  $L$).
We write the metric on this factor as $g = g(r)dr^2$
($g$ is now a function on ${\Bbb R}$).

The metric on the spacetime has thus  the form 
$g(r)dr^2 - v(r)dt^2 + r^2 d\sigma^2$
($d\sigma^2$ is the canonical
metric 
on $S^2$).

The third geometrical hypothesis is 
that the spacetime is 
empty (a vacuum), i.e. Ricci flat, leading 
to 
differential relations on 
the functions $g$ and $v$.
They imply that 
$g=  \frac{1}{1-(2m/r)}$, 
and
that $v g$ equals a constant (here one has to perform some computation).
This last constant must equal to  1, by the 
fourth geometrical hypothesis saying that 
the spacetime is 
asymptotically Minkowskian.

We have therefore, 
$L = ] 2m, +\infty [ \times {\Bbb R}$, 
endowed with the metric: $$ \frac{1}{1-(2m/r)} dr^2 -(1-(2m/r))dt^2$$
The warped product $L\times_{r^2}S^2$
is called the {\it Schwarzschild exterior} spacetime.

  It is natural to ask if   other solutions exist 
without our topological simplification hypotheses. This is
essentially equivalent to ask if the spacetime 
$L \times_{r^2} S^2$ admits non-trivial extensions. 
One easily sees that no such  {\it static}
extensions exist. 
However   non-trivial analytic (and thus Ricci flat)
 extensions actually exist. They (essentially) correspond  
to analytic  extensions of the Lorentz surface 
$L$.

The obvious one is given by
adding $L^-= 
 ]0, 2m[   \times {\Bbb R}$, endowed with 
the metric defined by the same formula. The warped 
product $L^-\times_{r^2} S^2$ is
called the {\it Schwarzschild black hole}.

It has been 
   observed (firstly by Lema\^itre, see for instance \cite{Mis}) 
 that the metric on $L \cup L^-$ admits an 
analytic extension to all
$]0, +\infty[ \times {\Bbb R}$.

Next, a larger extension $\hat{L}$ , which turns out to be
 ``maximal'', 
 was discovered by Kruskal. It can be described, 
 at a 
``topological level'' as follows. 
Endow ${\Bbb R}^2$ with coordinates $(x, y)$ and a 
Lorentz scalar
 product (at 0) $dx dy$. Then, $\hat{L}$
is the part of ${\Bbb R}^2$
defined by an inequality $xy > c(m)$, where 
$c(m)$ is a negative constant. The metric 
has the form $F(xy)dxdy$, where $F: ]c(m), +\infty[ \to {\Bbb R}$ is
an  analytic real function which tends to $\infty$ at $c(m)$.
(It turns out that a coordinate system where the metric has this 
form is unique
up to a linear diagonal transformation.)

From the form of the metric, 
the 
flow   $\phi^s(x, y) = (e^sx,e^{-s}y)$ 
acts isometrically on $\hat{L}$. This corresponds to the analytic
extension 
of the  Killing field $\frac{\partial} {\partial t}$
defined on $L$.

The time function
  $t$ on $L$ has the form $t(x, y)
= a\ln \frac{x}{y}$, where $a$ 
is a constant (which depends on the coordinate system).
 
The radius function $r$  looks like 
a Lorentz radius, indeed it has the form,
$r(x,y) = b(xy)+2m$,
for some analytic function $b: [c(m), +\infty[ \to [-2m, + \infty[ $,
with $b(0)= 0$. (A natural Lorentz radius 
for $({\Bbb R}^2, dxdy)$ is
$\sqrt { \vert  x y \vert }$).

Our initial surface $L$ is identified with the positive
quadrant 
$x, y >0$.

The warped product structure (determined by 
the flow $\phi^s$ on $\hat{L}- \{xy=0\}$) 
is physical on $xy>0$, and 
anti-physical
on $xy<0$. In fact, this structure 
is conformal to that determined by the polar 
coordinates on 
$({\Bbb R}^2, dxdy)$ (\S \S \ref{polar.1}, and \ref{polar.2}).

\subsubsection{Geodesic foliations.}

The factor $\hat{L}$ determines a geodesic foliation of the Kruskal 
spacetime $\hat{L}\times_{r^2} S^2$. 

The static 
structure (on 
$L \times_{r^2}S^2$) determines a geodesic foliation ${\mathcal F}$ with leaves 
$t= $ constant, or equivalently   $\frac{x}{y} =$ constant. 
Thus a  leaf   has
the form: $F = R \times S^2$, where
$R \subset \hat{L}$ is   
 a ray emanating from 0.

This foliation extends to 
  $(\hat{L}-0) \times S^2$ (and to the whole Kruskal spacetime 
$\hat{L} \times_{r^2} S^2$, as a singular 
geodesic foliation).

The causal character of a leaf $F$ is the same as that of the ray $R$.
 In particular,
lightlike leaves correspond to lightlike rays, i.e. the coordinate axis.

\subsubsection{Geodesics.}
To determine all the geodesics of 
 $L \times_{r^2}S^2$, one uses Theorem \ref{Maupertuis}
 which reduces the problem
to the calculation of the trajectories of
mechanical systems on the surface $L$ defined 
by the potentials $\frac{c}{r^2}$.

Now, since $L$ itself is a warped product, one applies 
Corollary \ref{mechanics.surface} to solve mechanical systems 
with potentials $\frac{c}{r^2}$
over it. This reduces to use the energy conservation, and 
solve the mechanical systems with potentials 
$ c_1\frac{1}{1-(2m/r)} + c_2\frac{1}{r^2} $
 on  
$({\Bbb R}, \frac{1}{1-(2m/r)}dt^2)$.

Proposition \ref{mechanics}
applies to these potentials
(considering  $L$ as a warped product),
 which allows one to 
fully explicit the geodesics.

\subsection{Motivations for anti-physical  warped products.}  

We think there is no reason to be troubled 
by anti-physical warped products. The 
adjective anti-physical   must not suggest that they are 
``non physical'', but  rather that they are ``mirror transform'' 
of physical ones (to be found?). This  clearly
happens in the case of polar coordinates in the Minkowski
space, where one sees how the anti-physical part of the 
GRW structure is dual to the physical one (\S \ref{polar.2}). A 
similar 
 duality holds between the interior and the exterior 
of the Schwarzshild spacetime. The exterior is  
  static, by  the existence of   a timelike 
Killing field, which 
becomes spacelike in the interior. The interior of a black 
hole is anti-physical.

Let us enumerate further (physical) motivations of anti-physical
warped products:

$\bullet$ With respect to the goal of constructing 
simple exact solutions, 
the calculations are formally 
the same, in the physical as well as in the  the anti-physical cases.  
So, one may calculate, and forget that it is an  anti-physical warped
product!


$\bullet$ As was said before, the abundance of symmetries 
leads to a warped product structure, but  actually,  large  
 symmetry groups involve
 anti-physical warped products.
For example, non-proper isometry groups lead to an anti-physical 
warped product (see for instance \cite{Ze.gafa}). Roughly speaking, non-proper 
means that 
the stabilizers are non-compact.  
Let us however say that only 
few exact solutions have non-proper isometry groups. 
It seems that this is the case, only for spaces of constant curvature and some 
gravitational plane waves.

$\bullet$ Finally,  it seems 
interesting to formulate
a  complexification trick which exchanges 
 anti-physical by   physical structures.  The very naive idea starts by 
 considering 
a Riemannian analytic submanifold $V$ in the  Euclidean space ${\Bbb R}^N$,
 taking its ``complexification''  $V^{\Bbb C}$ 
and then inducing  on  it the holomorphic metric 
of  ${\Bbb C}^N$, which as a real metric is pseudo-Riemannian. 
(The complexification is defined only locally but one may approximate 
by algebraic objects 
in order to get a global thing, see for instance \cite{HHL} for related 
questions).

\section{Big-bangs in anti-physical warped products}

Consider the example of polar coordinates around 0
in the Minkowski space ${\Bbb R}^{n, 1}$ (\S \ref{polar.2}).  When  an interior  point 
 approaches  
the light cone (and especially 0), the warping function collapses, 
and the warped product structure disappears. 
However, the spacetime itself  persists, beyond this 
``false big-bang''. It seems  interesting to know situations 
where a ``true big bang'' (i.e. a disappearing of the spacetime)
must follow from a disappearing of the warped product structure. The results
below provide an example of such a situation, but let us before
try to give a more precise definition.

\begin{definition} Let $({\mathcal L}, {\mathcal N})$ be a warped product structure  
on a pseudo-Riemannian manifold $M$. We say that it has an 
{\bf inessential big-bang} if there is an isometric embedding of $M$
in another pseudo-Riemannian manifold
$M^\prime$, as an open {\bf proper} subset, such that 
the shape vector field $\overrightarrow{n}$  of ${\mathcal N}$, is
non-bounded in some compact subset of $M^\prime$.

\end{definition}

In other words, we see $M$ as an open subset of $M^\prime$, then, an 
inessential big-bang holds if there is a compact $K $ in
$M^\prime$ such that  $\overrightarrow{n}$ is not bounded on
$K \cap M$. (Observe that one may speak of  bounded vector fields on compact sets
without any reference to metrics). We have the following result.

\begin{theorem} An analytic anti-physical  GRW structure 
with non-positively curved normal factor, has 
no inessential big-bangs.

\end{theorem}

 Let us give a purely mathematical essentially equivalent statement.

\begin{theorem} (\cite{Ze.gafa})
\label{big.bang.math}
Let $M$  be an  analytic Lorentz  manifold such that 
some open subset $U$ of $M$ is 
isometric
to a warped product $L \times_w N$,  where $N$ 
 (is Lorentzian and) has a constant non-positive  curvature
(i.e. $N$ is locally isometric to 
 the  Minkowski or the  anti de Sitter spaces).  
Then,  every point of $M$ has a neighborhood 
  isometric to 
a warped product of the same type. 
More precisely, if $M$ is simply connected, then the warped product structure 
 on $U$ extends to a local warped product (of the same type) on $M$.
\end{theorem}

Let us give another formulation in the vein of
detecting singularities of a  spacetime
from that of a warped product structure on it.

\begin{corollary}
Let $M$ be a simply connected manifold, and $U$ an open subset
 of $M$ endowed with an analytic Lorentz metric $g$. Suppose that $(U, g)$
 is a warped product as above, and let $x$ be a point in  the boundary of $U$. 
If the warping function $w$ tends to $\infty$ or 0 near $x$, then 
(not only the warped product structure, but also) the Lorentz metric 
$g$ does not extend analytically near $x$.  
\end{corollary}

\begin{remark} 
{\em 
The case of 
polar coordinates on the Minkowski space shows that the 
hypotheses  that the GRW structure 
is  anti-physical and the normal
 factor of non-positive curvature are 
necessary.

}
\end{remark}

In the sequel, we will give 
the proof of  Theorem \ref{big.bang.math}, and also details 
on the tools behind it,   especially  about 
 lightlike Killing fields.

\section{Proof of Theorem 
\ref{big.bang.math} }  

\label{proof.theorem}

\subsection{Beginning.}

\subsubsection{Trivial extension.}

Let ${\mathcal E}_c^{d+1}$ denote the simply connected complete Lorentz 
space of constant curvature $c$ (see for instance \cite{Wol} 
for more  details).

In the case $c \neq 0$, we assume that $d>1$, that is,  the dimension 
of the space is $\geq 3$.  In fact, in dimension 2, the sign of the curvature
is irrelevant.

Let $U = L \times_ w N$ be as in the statement of 
Theorem \ref{big.bang.math}. By hypothesis $N$ is  
locally isometric to 
  ${\mathcal E}_c^{d+1}$ for  some $c \leq 0$. We can  restrict $U$ so that 
$N$ becomes identified to an open subset 
of ${\mathcal E}_c^{d+1}$.

By  the trivial 
extension of isometries,  Fact \ref{isometry.extension}, 
 Isom $(N)$ acts on $U$. However, because $N$ is a
``small'' open subset of ${\mathcal E}_c^{d+1}$, Isom$(N)$ 
may be dramatically small, and for this, it is better to consider
 infinitesimal isometries, i.e. Killing vector fields. Indeed, 
like  isometries, Killing vector fields of $N$, trivially extend 
to $U$. Now the Killing algebra of $N$ (i.e. the algebra of
Killing fields) is the same as that of 
${\mathcal E}_c^{d+1}$ which we denote by 
${\mathcal G}_c^{d+1}$. Therefore there is an {\bf infinitesimal action} of 
${\mathcal G}_c^{d+1}$ on $U$, 
 i.e. a homomorphism which  for 
 $ X \in {\mathcal G}_c^{d+1}$  associates an element 
$\bar{X}$ of   the    Killing algebra of $U$.
  
Note that, for our purpose,  only the sign of $c$ 
is relevant, that is we can assume  $c= -1$, whenever 
$c <0$.

Recall that ${\mathcal G}_0^{d+1}$,  the Killing Lie algebra of 
the Minkowski space ${\mathcal E}_0^{d+1}$,  is  isomorphic
to a semi-direct product
  ${\Bbb  R}^{d+1} \rtimes o(1, d)$, 
and that the  Killing Lie algebra of 
the anti de Sitter space ${\mathcal E}_{-1}^{d+1}$ is 
 ${\mathcal G}_{-1}^{d+1} =o(2, d)$

\subsubsection{Analytic extension.}

\label{analytic.extension}

Henceforth, we will assume  that $M$ is simply connected and
 analytic (it suffices just to pass to the universal 
covering).
A classical result \cite{Nom}  states that an analytic Killing field
defined on an open subset extends as a Killing field to the
 whole 
of $M$.

 By individual extension of Killing fields, we get 
an infinitesimal analytic isometric action of 
${\mathcal G}$ on the whole of $M$.

However, this action does not  a priori  determine 
a regular foliation, namely,  the dimension of the orbits 
is not necessarily constant.

Let us first observe that the analyticity implies that
$d+1$, i.e. the dimension of the orbits of the points of $U$,  
is the generic   dimension of orbits, that 
is,  the  dimension is everywhere $\leq d+1$.
Indeed, 
if $X_1, \ldots, X_{d+1} \in {\mathcal G}$, then 
$\bar{X}(x) \wedge \ldots \wedge \bar{X}_{d+2}(x) = 0$ 
for $x \in U$, 
 and hence everywhere (of course, we implicitly 
assume that all our  spaces here are connected).

\begin{proposition}  Let 
${\mathcal G}= {\mathcal G}_c^{d+1}$ act infinitesimally isometrically 
on a Lorentz manifold $M$ (here 
$c$ is not assumed to be $\leq 0$), with a generic
orbit 
 dimension
$\leq d+1$. assume  that all (the restrictions 
of the metrics on) the orbits are 
non-degenerate. In the case $c >0$, assume further 
 that at least one orbit is of Lorentzian
type. Then, the ${\mathcal G}$-action 
determines a regular (i.e. with constant dimension) 
foliation, which is the normal foliation 
of a GRW structure.

\end{proposition}

\begin{proof}
Observe that an orbit is a ${\mathcal G}$-locally
 homogeneous space.
So, the proof of the proposition follows from 
Theorem \ref{action.criterion.2}
and from the 
  following classical fact.

\begin{fact}
\label{maximal.dimension}
If  a pseudo-Riemannian manifold of dimension
$\leq d$, has a Killing algebra of the same dimension
as that of a pseudo-Riemannian manifold of constant 
curvature and dimension 
$d$, then this manifold is 
necessarily of   dimension $d$ and has the same 
constant curvature.
\end{fact}

\begin{proof}
Recall that all the orthogonal algebras 
$o(p,q)$, with $p+q = d^\prime$ have the same 
dimension, which equals in particular dim $o(d^\prime)$.
Let $x$ be a point of the given pseudo-Riemannian manifold.
Its stabilizer algebra
can be identified to a subalgebra 
of some $o(p,q)$, with $p+q \leq d$. 
But by hypothesis, 
this  stabilizer  has a  dimension 
$\ge $ dim $o(d)$. It  follows that 
$p+q= d$, and that the stabilizer is 
$o(p,q)$ itself. One deduces,  in particular,  that  the dimension
 of the manifold equals $d$. 
To check that the curvature is constant, one 
  observes that $O(p,q)$ acts 
transitively 
on the space of  spacelike 2-planes at $x$.
\end{proof}

\end{proof}

\subsection{Lightlike Killing fields.}

The following notion will be  useful.

\begin{definition} 
A Killing field 
$X$ on a pseudo-Riemannian manifold 
is called geodesic (resp. lightlike) if 
$\nabla_XX= 0$ (resp. 
$\langle X,X\rangle  =0$).

\end{definition}

\begin{fact}
\label{lightlike.geodesic}

 A Killing field $X$ is geodesic 
iff,  it has geodesic orbits, iff, it has constant 
length
(i.e. $\langle X, X\rangle $ is constant). In particular 
a lightlike Killing field is geodesic.

\end{fact}

\begin{proof}
Let $\nabla$ denote the 
 Levi-Civita connection. 
Recall 
that a Killing field $X$ is characterized by the fact that 
$\nabla X$ is 
  antisymmetric, that is, 
  $\langle \nabla_YX,Z \rangle + \langle Y,\nabla_ZX \rangle  = 0$,
for any vector fields $Y$ and 
$Z$. In particular, 
 $ \langle \nabla_XX,Y \rangle + \langle \nabla_YX,X \rangle =0$,
and hence,  $ \langle \nabla_XX, Y \rangle  =
-(1/2)Y. \langle X,X \rangle $. Therefore, 
$\nabla_XX =0$ is equivalent 
to that $\langle X, X \rangle $ is constant.

\end{proof}

\subsubsection{Singularities.}

A geodesic Killing field with a somewhere 
non-vanishing length 
is non-singular (since it has a constant length). This fact 
extends to lightlike Killing fields 
on Lorentzian manifolds.

Indeed, near a singularity, the situation looks like that of the
 Minkowskian case. In this case, the Killing field
preserves (i.e. is tangent to)
the ``spheres''
around the singularity, but some of these
spheres are spacelike, contradiction!

As it is seen in this sketch of  proof, the fact actually  extends to 
non-spacelike Killing fields, i.e. 
$\langle X,X\rangle  \leq 0$:

\begin{fact} (\cite{BEM}, see also \cite{A-S} and 
\cite{Ze.math.Z}) 
 A non-trivial non-spacelike
Killing field on a Lorentz manifold is singularity free. 

\end{fact}

\subsubsection{Curvature.}

\begin{fact}
\label{curvature}
Let $X$ be a geodesic Killing field, then, 
for any vector $Y$, 
 $$\langle R(X,Y)X, Y\rangle  = \langle \nabla_YX, \nabla_YX\rangle $$ 

If $M$ is {\em Lorentzian} or 
{\em Riemannian}, and  $X$ is non-spacelike
(i.e. $\langle X,X\rangle  \leq 0$), then 
$\langle R(X,Y)X, Y\rangle  \geq 0$. 
In particular,   $Ric(X,X) \geq 0$, with 
equality (i.e.  everywhere  $Ric(X,X)= 0$),   iff, 
the direction of $X$ is parallel.

In the case $M$ is lorentzian and 
$X$ is lightlike, the curvature of any non-degenerate
2-plane containing $X$ is 
$\leq 0$.  

\end{fact}

\begin{proof}
Let $\gamma$ be a geodesic 
tangent to $Y$.
Consider the surface $S_\gamma$
obtained by saturating $\gamma$ by the flow of $X$, i.e. if 
$\phi^t$ is the flow of $X$, then $S_\gamma = 
\cup_t \phi^t(\gamma)$
(here we assume that $X$ is transversal to
$\gamma$).

Take a geodesic parameterization 
of 
$\gamma$, and continue to denote 
by  $Y$,  the vector field 
on $S_\gamma$,  obtained first, by  parallel 
translating along $\gamma$, and then, 
saturating by $\phi^t$ (along $S_\gamma$). 

We have: $X$ and $Y$ commute,  
$\nabla_YY =0$, 
and  $\nabla_XX= 0$ (since $X$ 
is 
a geodesic Killing field). It remains to estimate 
$\nabla_XY$ ($= \nabla_YX$).
We have $0= Y\langle X,X\rangle = 2\langle \nabla_YX,X\rangle $
(since $\langle X, X \rangle$ is constant by Fact \ref{lightlike.geodesic}) and 
$X\langle Y,Y\rangle  = 2 
\langle \nabla_XY,Y\rangle $, 
since by construction  $\langle Y, Y \rangle$ 
is constant along $S_\gamma$.
Therefore, 
$\nabla_XY$ ($= \nabla_YX$) is orthogonal
to $S_\gamma$.

One may restrict  consideration to  the case where $S_\gamma$ is non-degenerate,
since, if not, one  may  approximate $S_\gamma$ by 
non-degenerate $S_{\gamma_n}$, by choosing an
appropriate sequence of geodesics $\gamma_n$.

The previous calculation implies
 that $S_\gamma$ is intrinsically  flat, since 
the orthogonal projection of the ambient connection 
vanishes (all the covariant derivatives obtained 
from 
$X$ and $Y$ are orthogonal to $S_\gamma$)..

The  curvature equality
 follows from the Gau{\ss} equation.

Now, $\nabla_XY$ is orthogonal to $X$, and hence  it is
spacelike whenever  $X$ is non-spacelike and 
 $M$
is Riemannian or Lorentzian.

Recall that $Ric(X,X)$ equals the trace of the linear 
endomorphism $Y  \to A(Y) = R(X, Y)X$. Now, 
$\langle A(Y), Y\rangle  \geq 0$ implies that trace$(A) \geq 0$, and it is
also straightforward to see 
that if $Ric(X, X)= 0$, then 
$\nabla_YX  0$ is isotropic  for all $Y$.  We  have in addition that 
$\langle \nabla_Y X, X \rangle = 0$, and hence $\nabla_Y X $ is proportional 
to $X$. This is exactly the analytic translation of the fact that 
the direction field determined by    
$X$ is parallel.

Finally, the sectional curvature of the plane generated by $X$ and $Y$
is  $\frac {\langle R(X,Y)X,Y\rangle }{\langle X,X\rangle
 \langle Y,Y\rangle -\langle X,Y\rangle ^2}$, which has the 
opposite sign of 
$\langle R(X,Y)X,Y\rangle $.

\end{proof}

\subsubsection{The constant curvature case.}

Let ${\Bbb R}^{p,q}$ denote 
${\Bbb R}^n$ ($n = p+q$),  
endowed with the standard  form
 $Q = -x_1^2-\ldots-x_p^2 + x_{p+1}^2 + \ldots x_n^2$, 
 of signature 
$(p, q)$. 
A Killing field $X$ on 
${\Bbb R}^{p,q}$ is of the  form
$x \to Ax + a$, where $a \in {\Bbb R}^n$, 
and $A \in o(p, q)$.  Recall that 
$A \in o(p, q)$, iff, 
$AJ + JA^*=0$, where

$$J= \left(
\begin{array}{cc}
-I_p & 0 \\
0 & I_q
\end{array}
\right)$$


We have, $\nabla_XX = A^2$, 
and hence, $X$ is geodesic, iff, 
$A^2 = 0$.  

In the Lorentzian case (i.e. the Minkowski space),  
 $p =1$, the equation $A^2 =0$,  has no non-trivial
solution, that is, if $A \in o(1, p)$, and $A^2 =0$,
then,  $A= 0$. One may show this by a 
straightforward calculation, or by applying 
Fact \ref{curvature} to $S^{1,p}(+1)$,  which will
be   defined below.
It follows that a 
geodesic Killing field is  parallel, i.e. it  has  
the form $X: x \to a$, and it is
lightlike if furthermore  $a$ is isotropic.

In the non-Lorentzian case, non-trivial solutions of $A^2 =0$
exist.
Let us  consider the case of  ${\Bbb R}^{2,2}$. 
The standard form $Q$ is
equivalent to 
$Q^\prime  = dx dz + dy dt$.
Consider $\phi^s(x, y, z, t)=
(x, y, z+sx, t+sy)$. This is a one-parameter group
of orthogonal transformations of $Q^\prime$. Its infinitesimal
generator:

$$B= \left(
\begin{array}{cccc}
0&0&1&0  \\
  0&0&0&1  \\
 0&0&0&0 \\
  0&0&0&0 
 \end{array}
\right)$$


 satisfies $B^2 =0$. Thus, a conjugate $A$  of 
$B$ belongs to $o(2,2)$ and satisfies $A^2 =0$
A standard argument shows to that $o(2,2)$ is in fact generated 
by elements satisfying the equation $A^2= 0$. By the same argument one proves:

\begin{fact}
\label{rank}
 For $p \geq 2$, $q \geq 2$, $o(p, q)$
is generated (as a linear space) by its elements 
satisfying $A^2 =0$. (Note that  the condition on $p$ and $q$ 
is equivalent to 
that $o(p, q)$ has real rank $\geq 2$).

\end{fact}

Consider 
$X_c= S^{p,q}(c) = \{x / Q(x,x) =c \}$. Then, for
$c \neq 0$, 
$X_c$ is non-degenerate, and the metric  
on it has signature $(p, q-1)$ if $c>0$, and
signature $(p-1, q)$ if $c<0$. 
It has curvature $\frac{1}{c}$, and Killing algebra 
$o(p, q)$. The universal pseudo-Riemannian 
space of the same signature and curvature, 
is a cyclic (maybe trivial) covering of $X_c$.
The Killing algebra of the universal cover is the same
as that of $X_c$ (see \cite{Wol}).

A Killing field $A \in  o(p, q)$ is geodesic
(with respect to  $X_c$), iff,  $A^2 = \lambda I$, 
for some constant $\lambda$. It is 
lightlike, iff, $A^2 =0$.

For example, in the Riemannian case, i.e.
$p = 0$, solutions  of $A^2 = \lambda I$
in $o(n)$ exist exactly if $n$ is  even, which give
Hopf fibrations on odd dimensional spheres.

For the Lorentz case, we have, with the previous notations,   
${\mathcal E}_c^{d+1}= S^{1, d+1}(c)$, if $c >0$, and 
${\mathcal E}_c^{d+1}$ is the universal cover 
of $S^{2, d}(c)$, if $c<0$.

In particular, a solution of $A^2 = 0$ in $o(1, p)$
corresponds to a lightlike Killing field
on 
the de Sitter space ($= {\mathcal E}_c^{d+1}= S^{1, d+1}(c)$). 
  But, since this latter space 
is Lorentzian and has positive curvature, such a non-trivial
Killing field does not exist by Fact \ref{curvature}.
Summarizing: 

\begin{fact} The de Sitter space has no lightlike (or even geodesic) Killing fields.

A lightlike Killing  field on the Minkowski space is parallel with 
isotropic translation vector. 

The Killing algebra of the anti de Sitter 
space is generated, as a linear space, by its lightlike Killing 
fields.

\end{fact}

\subsection{End of the proof of Theorem \ref{big.bang.math}}

Observe that if $X \in {\mathcal G}_c^{d+1}$ is lightlike, 
as a Killing field on ${\mathcal E}_c^{d+1}$, then 
its trivial extension  $\bar{X}$, 
is a lightlike Killing field 
on $M$.

Suppose by contradiction that
there is a degenerate orbit $N_0$ of 
the ${\mathcal G}_c^{d+1}$-action. 

From \S \ref{analytic.extension}, 
$N_0$ has dimension $\leq d+1$.
Observe first  that dim$N_0 > 0$, since 
lightlike Killing fields are singularity free.

The metric on $N_0$ is positive non-definite, with kernel 
of dimension 1 (since the metric on $M$ is Lorentzian). This 
determines a 1-dimensional foliation 
${\mathcal F}$, called the  {\bf characteristic}
 foliation of $N_0$. 
The tangent direction of 
${\mathcal F}$ is the unique isotropic direction 
tangent to 
$N_0$.
It then follows that if $X$ is a
lightlike Killing field, then the restriction
of $\bar{X}$ to $N_1$ is tangent to ${\mathcal F}$
(equivalently, the flow of such a Killing 
field preserves individually the leaves of ${\mathcal F}$). 
Therefore, from Fact  \ref{lightlike.geodesic},  
  the 
 leaves of
${\mathcal F}$ are lightlike geodesics (in $M$).

\subsubsection*{The anti de Sitter case.} In the case $c < 0$, 
${\mathcal G}_c^{d+1}$
 is generated by lightlike Killing fields, and hence 
${\mathcal G}_c^{d+1}$ itself preserves individually the leaves of
 ${\mathcal F}$. Thus, by definition,  $N_0$ 
has dimension 1. However, it is known that there is no 
${\mathcal G}_c^{d+1}$-homogeneous space of dimension 1. 
This  is particularly easy to see in 
the  present situation. Indeed, here, 
${\mathcal G}_c^{d+1}$ preserves the affine structure of 
the lightlike geodesic $N_0$, and hence 
${\mathcal G}_c^{d+1}$ embeds in the Lie algebra of the affine group
of ${\Bbb R}$, which is impossible.

\subsubsection*{The flat case.}  If 
$N_0$ has dimension 1, we get a contradiction 
as in the anti de Sitter case. If not (i.e. dim$N_0 >1$), 
consider the (local) quotient space $Q = N_0 /{\mathcal F}$.
(The global quotient does not  necessarily exist, but 
because we deal with infinitesimal actions, we can restrict 
everything to a small open subset of $M$). The ${\mathcal G}_0^{d+1}$-action 
on $N_0$  factors through a faithful  action
of $o(1,d)$ ($= {\mathcal G}_0^{d+1}/ {\Bbb R}^{d+1}$) on $Q$. 

Observe that $Q$ inherits  a natural Riemannian 
metric. Indeed, the Lorentz metric restricted  to $N_0$ 
is positive degenerate, with kernel
$T{\mathcal F}$.  But ${\mathcal F}$ is parameterized by any lightlike field 
$X \in {\mathcal G}_0^{d+1}$ (this 
is the meaning of the fact that the flow of $\bar{X}$ preserves 
individually the 
leaves of ${\mathcal F}$).  Therefore the projection 
of this metric on $Q$ is well defined.

This metric is invariant by 
the $o(1, d)$-action. As in the
proof of Fact \ref{maximal.dimension},  since dim$Q \leq d$, we 
have   dim$Q= d$, and furthermore,  $Q$ has constant curvature. Also, we
 recognize from
the list of Killing algebras of constant curvature manifolds that 
$Q$ has constant negative curvature, i.e. $Q$
is a hyperbolic space.

It then follows that dim$N_0 = d$, and in particular that 
the orbits of ${\mathcal G}_c^{d+1}$ determine  a  regular
foliation near $N_0$.

Now, the contradiction lies in the fact that 
$Q$ is hyperbolic, but the  analogous quotient 
for generic leaves of the ${\mathcal G}_0^{d+1}$-action, are 
flat. 
More precisely, let $X \in {\Bbb R}^{d+1} 
\subset {\mathcal G}_0^{d+1}$
be a translation timelike Killing field. 
Consider $M^\prime$ the (local)  space of orbits of $X$
(instead of the whole of  $M$, we take a small
open subset intersecting $N_0$, where everything is topologically trivial).
The ${\mathcal G}_0^{d+1}$-orbit foliation projects to a foliation 
${\mathcal G}^\prime$ of $M^\prime$. For example, $Q$ is a 
leaf of ${\mathcal G}^\prime$ which is 
just the projection of 
$N_0$. In fact, as in the case of $Q$, the
projection of the metric on 
the ${\mathcal G}_0^{d+1}$-orbits endows the 
 leaves of ${\mathcal G}^\prime$
  with a Riemannian metric.  Now, a generic 
leaf of ${\mathcal G}^\prime$ is (locally) isometric
to the quotient of the Minkowski space 
 ${\Bbb R}^{d,1}$ by 
a timelike translation flow, which is thus a  Euclidean space (of
dimension $d$). But the leaf $Q$ is hyperbolic
which contradicts the obvious continuity 
(in fact the  analyticity) 
of the 
  leafwise 
metric of ${\mathcal G}^\prime$. $\diamondsuit$


\end{document}